\documentclass[10pt,reqno]{amsart}
\usepackage{graphicx}

\usepackage{float}
\usepackage{listings}

\usepackage{amsmath}
\usepackage{hyperref}
\usepackage{cleveref}

\theoremstyle{plain}
\newtheorem{lem}{Lemma}
\newtheorem{theorem}{Theorem}

\theoremstyle{definition}
\newtheorem{Def}{Definition}

\theoremstyle{remark}
\newtheorem{remark}{Remark}
\newtheorem{example}{Example}
\newtheorem{conjecture}{Conjecture}

\newcommand{\tD}{\mathtt 2}
\newcommand{\tL}{\mathtt 1} 
\newcommand{\tO}{\mathtt 0}

\begin{document}

\title[Maximum order complexity in Zeckendorf base]{Maximum order complexity of the sum of digits function in Zeckendorf base and  polynomial subsequences}

\author{Damien Jamet \and Pierre Popoli \and
        Thomas Stoll 
}

\begin{abstract}
Automatic sequences are not suitable sequences for cryptographic applications since both their subword complexity and their expansion complexity are small, and their correlation measure of order 2 is large. These sequences are highly predictable despite having a large maximum order complexity. However, recent results show that polynomial subsequences of automatic sequences, such as the Thue--Morse sequence, are better candidates for pseudorandom sequences. A natural generalization of automatic sequences are morphic sequences, given by a fixed point of a prolongeable morphism that is not necessarily uniform. In this paper we prove a lower bound for the maximum order complexity of the sum of digits function in Zeckendorf base which is an example of a morphic sequence. We also prove that the polynomial subsequences of this sequence keep  large maximum order complexity, such as the Thue--Morse sequence. 

\end{abstract}

\maketitle


\section{Introduction}

Feedback shift register (FSR) sequences are used for many cryptographic applications such as pseudorandom number generators for stream cipher cryptosystems, see~\cite{golomb1967}. A binary $n$-stage feedback shift register (FSR) is a mapping $\mathfrak{F}$ from $\mathbb{F}_2^n$ to $\mathbb{F}_2^n$ of the form \begin{align*}
\mathfrak{F}: (x_0,x_1,\ldots,x_{n-1}) \mapsto (x_1,x_2,\ldots,x_{n-1},f(x_0,x_1,\ldots,x_{n-1})),
\end{align*}
where $f$ is a function from $\mathbb{F}_2^n$ to $\mathbb{F}_2$. Consider the binary sequence $\mathcal{S}=(s_i)_{i\geq 0}$ whose first $n$ terms are given and the remaining terms are uniquely determined by the recurrence relation \begin{align*}
s_{i+n}=f(s_i,\ldots,s_{i+n-1}), \quad i\geq 0.
\end{align*} The sequence $\mathcal{S}$ is called the output sequence of the FSR. An output sequence of a short FSR is considered weak in cryptographic applications.  In order to determine this shortness for an infinite sequence $\mathcal{S}$, Jansen~\cite{jansen1989,jansen1991} introduced the notion of $N$th \textit{maximum order complexity of} $\mathcal{S}$, denoted as $M(\mathcal{S},N)$, which is the length of the shortest FSR that generates the first $N$ elements of $\mathcal{S}$. If the mapping $\mathfrak{F}$ is a linear transformation, then the FSR is called a linear feedback shift register (LFSR). This leads to the notion of the \textit{linear complexity profile} of $\mathcal{S}$ denoted as $L(\mathcal{S},N)$. We have obviously $M(\mathcal{S},N)\leq L(\mathcal{S},N)$. The sequence $(M(\mathcal{S},N))_N$ is called \textit{the maximum order complexity profile} of $\mathcal{S}$ and $M(\mathcal{S})=\sup_{N\geq 1} M(\mathcal{S},N)$ is called \textit{the maximum order complexity} of $\mathcal{S}$.

A sequence $\mathcal{S}$ is said to have a perfect linear complexity profile if \begin{align*}
L(\mathcal{S},N) =\left\lceil \frac{n}{2} \right\rceil, \quad n\geq 1.
\end{align*} 

An explicit construction of sequences with perfect linear complexity profile is given in~\cite{AHN2020}. These are known under the name of apwenian sequences (see \cite{MW2021}). Diem~\cite{diem2012} observed that these sequences and sequences based on function expansion into expansion series can be efficiently computed from relatively short sequences. This leads to the notion of \textit{expansion complexity}, denoted as $E(\mathcal{S},N)$, see \Cref{Expansion_complexity} below for more details. It was proved by M\'erai, Niederreiter and Winterhof~\cite{merainiederreiterwinterhof2017} that the expansion complexity and the linear complexity of an infinite sequence satisfy $E(\mathcal{S},N)\leq L(\mathcal{S},N)+1$.

A natural question is to ask about the mutual relationship between the maximum order complexity and the expansion complexity. It is known that these two complexities have expected behavior $\log N$ for the maximum order complexity, see~\cite{jansen1989,niederreiterxing2014}, and $\sqrt{N}$ for the expansion complexity, see \cite[Theorem 6.2]{MW2021} and \cite{merainiederreiterwinterhof2017}. This suggests that maximum order complexity is smaller than the expansion complexity. However, the Thue--Morse sequence has a large maximum order complexity and a bounded expansion complexity. Thus in order to determine the unpredictability of a sequence, both the expansion complexity and the maximum order complexity are useful and needed. 

In recent years, research focused in particular on automatic sequences, i.e. sequences that are generated by a deterministic finite automaton. These sequences are not suitable for cryptographic applications as we will state later. Nevertheless, their pseudorandom behavior changes radically when the sequence is rarefied along special subsequences such as polynomial subsequences. One can extend the definition of automatic to morphic sequences which is an again larger class. 

The aim of the present article is to study pseudorandomness of some morphic sequences and its polynomially rarefied subsequences. We first introduce several measures of complexity (\Cref{Measures_of_complexity}) and illustrate their behavior with a simple example. We then introduce automatic and morphic sequences (\Cref{Automatic-morphic}) and give an overview about the results known for their measures of complexity. We then indicate our main results of this paper (\Cref{Main_results}).

\subsection{Measures of complexity} \label{Measures_of_complexity} Let us first introduce some measures of complexity that we will focus on in the present paper. 
\begin{Def}[Maximum order complexity]
Let $N$ be a positive integer with $N \geq 2$, and $\mathcal{S}=\left(s_n \right)_{n\geq 0}$ be a sequence over $\{0,1\}$ with $(s_0,\ldots,s_{N-2})\neq (a,\ldots,a)$ for $a=0$ or $1$. The  \emph{$N$th maximum order complexity} $M(\mathcal{S},N)$ is the smallest positive integer $M$ such that there is a polynomial $f(x_1,\ldots,x_M) \in \mathbb{F}_2[x_1,\ldots,x_M]$ with \begin{align*}
s_{i+M}=f(s_i,\ldots,s_{i+M-1}), \quad 0\leq i\leq N-M-1.
\end{align*}
If $s_i=a$ for $i=0,\ldots,N-2$, we define $M(\mathcal{S},N)=0$ if $s_{N-1}=a$ and $M(\mathcal{S},N)=N-1$ else. 
\end{Def}

A sequence with small maximum order complexity is not suitable for cryptographic applications. Indeed, such a sequence can be built with relatively short blocks of consecutive terms. However, a sequence with large maximum order complexity can still be predictable, as it is known for the Thue--Morse sequence, see~\cite{mauduitsarkozy1998,sunwinterhof2019}. Note also that the largest possible order of magnitude of $M(\mathcal{S},N)$ is $N$ and the expected value of $M(\mathcal{S},N)$ is $\log N$, see~\cite{jansen1989,niederreiterxing2014}.

The maximum order complexity has been studied by several authors, for general results see~\cite{isikwinterhof2017,jansen1989,jansen1991,niederreiterxing2014,topuzogluwinterhof2007}  or~\cite{popoli2020,sunwinterhof2019bis,sunwinterhof2019,sunzenglin2020} for applications to some particular sequences such as the Thue--Morse sequence or the Rudin--Shapiro sequence. From a computational perspective, Jansen \cite[Proposition 3.17]{jansen1989} showed how Blumer's DAWG (Direct Acyclic Weighted Graph) algorithm~\cite{blumerblumerhausslerehrenfeuchtchenseiferas1985} can be used to compute the maximum order complexity in linear time and memory. In the last section of this paper, we use DAWG algorithm to formulate some conjectures (Section~\ref{secConjecture}).

\medskip

Diem~\cite{diem2012} introduced the expansion complexity of a sequence as follows. 

\begin{Def}[Expansion complexity] \label{Expansion_complexity}
Let $N$ be a positive integer, $\mathcal{S}=\left(s_n \right)_{n\geq 0}$ be a sequence over $\{0,1\}$ and $G_{\mathcal{S}}(x)$ its generating function defined by \begin{align*}
G_{\mathcal{S}}(x)=\sum \limits_{i\geq 0}s_ix^i.
\end{align*} The \emph{$N$th expansion complexity} $E(\mathcal{S},N)$ is defined as the least total degree of a nonzero polynomial $h(x,y) \in \mathbb{F}_2[x,y]$ with \begin{align*}
h(x,G_{\mathcal{S}}(x)) \equiv 0 \pmod{x^N},
\end{align*} 
if $s_0,\ldots,s_{N-1}$ are not all equal to $0$, and $E(\mathcal{S},N)=0$ otherwise. 
\end{Def}

A truly random sequence has expected value $E(\mathcal{S},N)$ of order $N^{1/2}$, see \cite{MW2021}.

Binary, and more generally, $p$-ary automatic sequences are not good pseudorandom sequences since their expansion complexity satisfies $E(\mathcal{S},N)<+\infty$ by Christol's theorem, see~\cite{christol1979}. In fact, automatic sequences generated by a finite automaton do not have sufficiently many different factors of a fixed length. 

The subword complexity, denoted by $p_{\mathcal{S}}(k)$, counts how many different subwords of fixed length $k$ the sequence contains. It can again serve as a measure of pseudorandomness. A sequence $\mathcal{S}$ over $\{0,1\}$ is \textit{normal} if for every $k\geq 1$ and for any block of length $(b_0,\ldots, b_{k-1})$ $\in \{0,1\}^k$, we have \begin{align*}
\lim_{N \to \infty } \frac{1}{N}\,\text{Card}\{i<N:\; s_i=b_0,\ldots,s_{i+k-1}=b_{k-1} \}=\frac{1}{2^k}.
\end{align*}

One can associate to a  binary sequence a \textit{symbolic dynamical system} based on the shift on $\{0,1\}^{\mathbb{N}}$ . The \textit{topological entropy} of such a dynamical system is $\lim_{k\rightarrow +\infty} \frac{\log_2 p_{\mathcal{S}}(k)}{k}$, where $\log_2$ denotes the base 2-logarithm. A sequence with $0$ topological entropy is said to be \textit{deterministic} and it should not provide a pseudo-random sequence since such a sequence has a small subword complexity, see~\cite{queffelec1987}.

\subsection{Automatic and morphic sequences}\label{Automatic-morphic}

Let $k\geq 2$ and $\Sigma$ be a finite alphabet. We denote by $\Sigma^*$ the set of all finite or infinite words over $\Sigma$. A morphism $f$ is a mapping $f:\Sigma^* \rightarrow \Sigma^*$ satisfying $f(uv) = f(u)f(v)$ for all words $u$ and $v$. A morphism is said to be \textit{$k$-uniform} if $|f(x)|=k$ for all $x\in \Sigma$, where $|x|$ is the length of the word $x$. If there is $b \in \Sigma$ such that $f(b)$ starts with $b$, i.e. $f(b)=bu$ for some $u\in \Sigma^{*}$, we say that $f$ is \textit{$b$-prolongable}. In this case, let us denote $f^{\omega}(b)$ the fixed point of $f$ given by $f^{\omega}(b)=bf(u)f^2(u)\ldots$

\begin{Def}[Automatic and morphic sequences] Let $\mathcal{S}$ be a sequence over an alphabet $\Sigma$. $\mathcal{S}$ is \textit{morphic} if $\mathcal{S}=\pi ( f^{\omega}(b))$ where $f:\Sigma^{*} \rightarrow \Sigma^{*}$ is a $b$-prolongeable morphism and $\pi:\Sigma\rightarrow \Delta$ is a morphism, named coding. $\mathcal{S}$ is said \textit{$k$-automatic} if $f$ is $k$-uniform. 
\end{Def}

For equivalent definitions of  automatic sequences, we refer to~\cite{alloucheshallit2003}. An emblematic example of  automatic sequence is the Thue--Morse sequence, or Prou\-het--Thue--Morse sequence, independently introduced by Thue~\cite{Thue1912}, Morse~\cite{morse1921} and Prouhet sixty years before~\cite{prouhet1851}  . This sequence appears in many different fields of mathematics, see~\cite{mauduit2001}, and is still an important object of recent work. 

\begin{Def}[Thue--Morse sequence]
For an integer $n\geq 0$, we write 

$n=\sum_{i \geq 0}\varepsilon_i2^i$ with $\varepsilon_i \in \{0,1\}$ for all $i$ and $(n)_2=\cdots\varepsilon_1 \varepsilon_0$. The binary sum-of-digits of $n$ equals $s_2(n)=\sum_{ i \geq 0} \varepsilon_i$. The \textit{Thue--Morse sequence} $\mathcal{T}=(t(n))_{n\geq 0}$ is defined by $t(n)=s_2(n) \bmod 2$. 
\end{Def}

The Thue--Morse sequence is $2$-automatic and the corresponding morphism is \begin{align*} f:\left\{
    \begin{array}{lll}
         0 \mapsto 01 \\  1 \mapsto 10. \\ 
    \end{array}
\right. 
\end{align*}

The Thue--Morse sequence has a large maximum order complexity, see~\cite{sunwinterhof2019}, but this sequence is far from being pseudorandom. Indeed, it is well-known that $E(\mathcal{T},N)\leq 5$ for all $N$ since $h(x,y)=(x+1)^3y^2+(x+1)^2y+x$ satisfies $h(x,G_{\mathcal{T}}(x))=0$ where $G_{\mathcal{T}}(x)$ is the generating function of $\mathcal{T}$, see \cite[Example 12.1.12]{alloucheshallit2003}. Also, its subword complexity is small, $p_{\mathcal{T}}(k)\ll k$ (see \cite[Corollary 10.3.2]{alloucheshallit2003} for a general result for all automatic sequences). 

The behavior of this sequence regarding the defined pseudorandomness measures changes when this sequence is rarefied along specific subsequences. Drmota, Mauduit et Rivat~\cite{drmotamauduitrivat2019} showed that the Thue--Morse sequence along squares is normal and Moshe~\cite{moshe2007} showed that the polynomial subsequences of degree $d\geq 2$ of the Thue--Morse sequence have an exponential subword complexity. Sun and Winterhof~\cite{sunwinterhof2019bis} and Popoli~\cite{popoli2020} showed that the maximum order complexity of the Thue--Morse sequence along polynomial subsequences remains comparatively large. 
A large maximum order complexity is desired, however, it should be noted that it should be not too large since the correlation measure of order 2 gets large in that case (see \cite[Proposition 3.1]{jansen1989}).

Furthermore, these polynomial subsequences are no longer automatic sequences, see \cite[Theorem 6.10.1]{alloucheshallit2003}, and their expansion complexity is no longer bounded by Christol's theorem. These statements mean that the polynomial subsequences of the Thue--Morse sequence are better candidates for cryptographic applications than the original sequence. 

In order to generalize the investigation to other morphic sequences and sum of digits function, let us introduce the Zeckendorf base sum of digits function, which is related to the Fibonacci sequence (see~\cite[Theorem 3.8.1]{alloucheshallit2003},~\cite{zeckendorf1972}). 

\begin{Def}[Zeckendorf base]
Let $\mathcal{F}=(F_n)_{n\geq 0}$ be the Fibonacci sequence with initial values $F_0=0$, $F_1=1$ and $F_{n+2}=F_{n+1}+F_n$ for all $n\geq 0$. Let us denote $\varphi=(1+ \sqrt{5})/2$ the golden ratio. Each integer $n$ can be represented uniquely by \begin{align*} n=\sum_{i \geq 0} \varepsilon_i(n) F_{i+2},
\end{align*} with $\varepsilon_i(n) \in \{0,1\}$ and $\varepsilon_i(n)\varepsilon_{i+1}(n)=0$ for all $i\geq 0$. 
\end{Def}

Since $F_n=\lfloor \frac{\varphi^n}{\sqrt{5}}+\frac{1}{2} \rfloor$ for $n\geq 0$, the index $n(a)$ of the largest Fibonacci number that is not greater than an integer $a>1$ is $\lfloor \frac{\log(a \sqrt{5}+1/2)}{\log(\varphi)} \rfloor$. We call $n(a)-1$ the \textit{length} of $a$.

Thus we can define the sum of digits function in Zeckendorf base by  $s_Z(n)=\sum_{i \geq 0} \varepsilon_i(n)$, see~\cite{drmotamullnerspiegelhofer2018,shallit2020,stoll2013} or sequence A007895 in \textit{On-Line Encyclopedia of Integer Sequences} (OEIS) for discussions on this sequence. We denote by $\mathcal{S}_Z$ the sequence defined by $\mathcal{S}_Z=\left(s_Z(n)\mod 2\right)_{n\geq 0}$, an analog of the Thue--Morse sequence. It is known that $\mathcal{S}_Z$ is a morphic sequence, see for example \cite[p.14]{bruyere1995} or \cite[Examples 7.8.2, 7.8.4]{alloucheshallit2003}. An example of the corresponding morphism and coding is $$f:\left\{
    \begin{array}{lll}
         a \mapsto ab \\  b \mapsto c \\ c \mapsto cd \\ d \mapsto a
    \end{array}
\right.  \quad \text{and} \quad  g:\left\{
    \begin{array}{lll}
         a \mapsto 0 \\ b \mapsto 1 \\ c \mapsto 1 \\ d \mapsto 0.
    \end{array}
\right. $$
Thus, the sequence modulo $2$ is obtained by iterating $f$ on the letter $a$ and recoding with $g$ : $\mathcal{S}_Z=\tO\tL\tL\tL\tO\tL\tO\tO\tL\tO\tO\tO\tL\tL\tO\tO\tO\tL\tO\tL\tL\ldots$ This sequence is not automatic, see \cite[Remark 1]{drmotamullnerspiegelhofer2018}. Indeed the authors show that the $k$-kernel is infinite for every $k$, see~\cite{alloucheshallit2003} for more details on equivalent definitions of automatic sequences. This implies that the expansion complexity of $\mathcal{S}_Z$ is not bounded, since this sequence is not automatic, and motivates our study of this sequence. In fact, morphic non-automatic sequences are better candidates for cryptographic applications regarding this complexity. However morphic sequences have subword complexity at most bounded by a quadratic function, see \cite[Corollary 10.4.9]{alloucheshallit2003}, and their associated dynamical system has $0$ topological entropy, see \cite[Corollary V.20]{queffelec1987} for an alternative proof. 

\begin{example}[Carry propagations in Zeckendorf base] \label{Zeckendorf}
In Zeckendorf base, the carry propagations work in a different way than in the usual $q$-base with $q\in \mathbb{N}$ and $q\geq 2$. Indeed, since we cannot have two consecutive $\tL$-bits, the carry is ``transversal'' by the Fibonacci recurrence relation and if we have two $\tL$-bits on the same column we use the formulas $2F_j=F_{j+1}+F_{j-2}$ for all $j \geq 2$ in order to get an admissible expansion. For example, $5=F_5$, $6=F_5+F_2$ and $5+6=11=F_6+F_4$ as shown by the following calculation:\[\begin{tabular}{ccccccll}
& & $\tL$ & $\tO$ & $\tO$ & $\tO$\\
+& & $\tL$ & $\tO$ & $\tO$ & $\tL$\\ \hline
=&&$\tD$ & $\tO$ & $\tO$ & $\tL$ \\ \hline
=&$\tL$&$\tO$ & $\tO$ & $\tL$ & $\tL$ \\ \hline
=& $\tL$&$\tO$ & $\tL$ & $\tO$ & $\tO$
\end{tabular}\]

In this kind of calculation, we first do not take into consideration the restriction of non-adjacent terms then we normalize to obtain the Zeckendorf representation with non-adjacent terms. This simple example shows that the carry propagation is on both sides of the expansion since $2F_j=F_{j+1}+F_{j-2}$ for all $j\geq 2$. We will see later that this specific carry propagation, compared to the usual $q$-base carry propagation, will play an important role in our investigation. 
\end{example}

As the sum of digits function for the usual $q$-base,  $s_Z$ is a subadditive function, i.e. for any integers $n_1,n_2\geq 0$ we have $s_Z(n_1+n_2)\leq s_Z(n_1)+s_Z(n_2)$, see \cite[Proposition 1]{stoll2013}. Consider the Zeckendorf expansion of $n_1$ and $n_2$,\begin{align*}
n_i=\sum \limits_{j=k_i}^{K_i}\varepsilon_j^{(i)}F_{j+2}, \qquad i=1,2;
\end{align*} with $\varepsilon_{k_i}^{(i)}=\varepsilon_{K_i}^{(i)}=1$. Without loss of generality suppose that $k_1\leq k_2$. We say that $n_1$ and $n_2$ are \textit{non-interfering} if $k_2-K_1\geq 2$. This means that non-interfering integers have digital blocks that do not overlap and no normalization after the addition has to be executed. We therefore have
\begin{equation}\label{noninterFibo}
s_Z(n_1+n_2)=s_Z(n_1)+s_Z(n_2).
\end{equation}

\begin{remark}\label{remark} For usual $q$-base we have an analogous definition. Let $a,b$ be integers such that $a<q^{\ell}$ for some $\ell\geq 0$, thus we have the identity \begin{align} \label{q-base}
s_q(a+q^{\ell}b)=s_q(a)+s_q(b),
\end{align}where $s_q$ denotes the sum of digits function in $q$-base. For Zeckendorf base there are two major differences with respect to the usual $q$-base. First, we have a transversal carry propagation and this translates into the condition $k_2-K_1\geq 2$ instead of $k_2-K_1\geq 1$ that we would have for the usual $q$-base. Secondly, the identity \eqref{q-base} is not true in general when we replace $q^{\ell}$ by $F_{\ell}$ and $s_q$ by $s_Z$. We will use  Lucas numbers to replace the usual shift $q^{\ell}$.
\end{remark}

\subsection{Main results}\label{Main_results}

In the following we will determine a lower bound for the maximum order complexity of $\mathcal{S}_Z=(s_Z(n) \mod 2)_n$ and for $\mathcal{S}_{Z,P}=(s_Z(P(n)) \mod 2)_n$ with $P \in \mathbb{Z}[X]$ a monic polynomial such that $P(\mathbb{N}_0) \subset \mathbb{N}_0$. In what follows, we use Vinogradov's notation $f\ll g$ is defined as $|f|\leq c|g|$ for some $c>0$.

\begin{theorem} \label{Main_theorem_linear}
There exists $N_0>0$ such that for all $N>N_0$ we have \begin{align*}
M(\mathcal{S}_Z,N) \geq  \frac{1}{\varphi+\varphi^3}N+1.
\end{align*}

\end{theorem}
Since the correlation measure is bounded below by the maximum order complexity (see \cite{MW2021}), Theorem \ref{Main_theorem_linear} shows that $\mathcal{S}_Z$ is \textit{non-random} in this specific respect. We mention also the recent result of Shutov \cite{Sh20} that shows that the autocorrelation of order 2 with lag equal to 1 of $(-1)^{s_Z(n)}$ is large. This serves as motivation to look at subsequences, that we address in following result.

\begin{theorem} \label{Main_theorem_polynomial}
Let $P(X) \in \mathbb{Z}[X]$ be a monic polynomial of degree $d\geq 2$ with $P(\mathbb{N}_0) \subset \mathbb{N}_0$. Then $\mathcal{S}_{Z,P}$ satisfies \begin{align*}
M(\mathcal{S}_{Z,P},N)\gg N^{1/(2d)}, \qquad N \to \infty,
\end{align*}where the implied constant only depends on $P$.
\end{theorem}

We will discuss the heuristics that supports the real growth of the maximum order complexity in the final part of the paper (Section \ref{secConjecture}). In particular, we conjecture that the lower bound is indeed the actual growth of the maximal order complexity.

The rest of the paper is structured as follows. In Section $2$, we prove \Cref{Main_theorem_linear} and \Cref{Main_theorem_polynomial} where we first prove the linear case and the particular case $P(X)=X^d$ for $d\geq 2$ before tackling the general case. We conclude the paper with some final remarks (\Cref{sec_final_remarks}) about possible generalizations of this result to other numerations systems and some heuristically supported conjectures.

\section{Sum of digits function in Zeckendorf base}

In order to prove a lower bound for the \textit{Nth maximum order complexity} of a sequence, we use a tool from~\cite[Proposition 3.1]{jansen1989}.
\begin{lem}[\cite{jansen1989}]\label{Method_Jansen} Let $\mathcal{S}$ be a sequence over $\mathbb{F}_2$ of length $n$. Let $k$ be the length of the longest subsequence of $\mathcal{S}$ that occurs at least twice with different successors. Then $\mathcal{S}$ has maximum order complexity $k+1$.
\end{lem}

The proofs of \Cref{Main_theorem_linear} and \Cref{Main_theorem_polynomial} are split into two parts. Let $N$ be a sufficiently large integer. First we built two subsequences of  the sequence $\mathcal{S}$ of same length $L(N)$, depending on $N$, which coincide by using non-interfering terms. Then, for these subsequences, we look for different successors, by studying precisely the involved carry propagations. Thus, by \Cref{Method_Jansen}, we then will have $M(\mathcal{S},N)\geq L(N)+1$.

\subsection{Linear case}
We first study the maximum order complexity of $\mathcal{S}_Z$. 

\begin{lem}\label{Maximum_order_complexity_linear}
Let $\ell \geq 2$, for all $0 \leq n<F_{\ell}$ we have  \begin{align*} s_Z(n+F_{\ell+1})&=s_Z(n+F_{\ell+2}), \\ s_Z(F_{\ell}+F_{\ell+1}) \bmod 2&\neq s_Z(F_{\ell}+F_{\ell+2}) \bmod 2.
\end{align*}
\end{lem}
\begin{proof}
If $n<F_{\ell}$, the terms $n$ and $F_{\ell+1}$, respectively, $n$ and $F_{\ell+2}$ are non-interfering, see (\ref{noninterFibo}). The second line follows from $s_Z(F_{\ell}+F_{\ell+1})=1$ and $s_Z(F_{\ell}+F_{\ell+2})=2$.
\end{proof}

We are now able to prove \Cref{Main_theorem_linear}.
\begin{proof}[Proof of \Cref{Main_theorem_linear}]
We choose $\ell\geq 2$ such as $F_{\ell}+F_{\ell+2} \leq N < F_{\ell+1}+F_{\ell+3}$. This implies $N\geq F_2+F_4=4$. Then by \Cref{Maximum_order_complexity_linear}, $(s_Z(n+F_{\ell+1}))_{0\leq n <F_{\ell}}$ and $(s_Z(n+F_{\ell+2}))_{0\leq n <F_{\ell}}$ are two subsequences of same length with different successors. Thus by \Cref{Method_Jansen}, we have $M(\mathcal{S}_Z,N)\geq F_{\ell}+1$. Furthermore, since $\lim_{n\to \infty} \frac{F_{n+k}}{F_n}= \varphi^k$, for $k\geq 1$, we have $N<(\varphi+\varphi^3)F_{\ell}$ for $\ell$ large enough. Thus, the theorem is proved. 
\end{proof}

\subsection{Monomial subsequences}

We now study the sequence $\mathcal{S}_Z$ along polynomial values. In a classical $q$-base, $xq^j$ is a shift of length $j$ for the expansion of $x$. In Zeckendorf base, the Fibonacci numbers do not have this property since a power of a Fibonacci number is in general not a Fibonacci number (note that the only exceptions of pure powers that are Fibonacci numbers are $1,8$ and $144$, see~\cite{bugeaudmignottesiksek2006}). The Lucas numbers are an interesting analogue of powers of $q$ in $q$-base.
\begin{Def}[Lucas numbers]
Let $\mathcal{L}=(L_n)_n$ be the sequence defined by $L_0=2$, $L_1=1$ and $L_{j+2}=L_{j+1}+L_j$ for all $j \geq 0$. 
\end{Def}

We have for all $j\geq 1$, the basic relation $L_j=F_{j+1}+F_{j-1}$, which means that the expansion of Lucas numbers in Zeckendorf base is simple. Moreover we have the following formulas.

\begin{lem}[\cite{stoll2013}] \label{Lucas}
For all $k\geq \ell \geq 0$ and $h\geq 0$, we have \begin{enumerate}
\item $L_kL_{\ell}=L_{k+\ell}+(-1)^{\ell}L_{k-\ell}$,
\item $L_kF_{\ell}=F_{k+\ell}-(-1)^{\ell}F_{k-\ell}$, 
\item $L_k^h=\sum \limits_{i=0}^{\frac{h-1}{2}}\binom{h}{i}(-1)^{ik}L_{(h-2i)k},\qquad h$ odd, 
\item $L_k^h=\sum \limits_{i=0}^{\frac{h}{2}-1}\binom{h}{i}(-1)^{ik}L_{(h-2i)k}+\binom{h}{h/2}(-1)^{hk/2},\qquad h$ even. 
\end{enumerate}

Furthermore, let $m> 0$, there are non-adjacent terms $\varepsilon_{-(2u+1)}(m),\ldots$,

 $\varepsilon_{2u+v}(m) \in \{0,1\}$, $\varepsilon_{-2u-1}(m)=\varepsilon_{2u+v}(m)=1$ for some $u,v$ integers depending only on $m$ such that for all $k\geq 2u+3$,
\begin{align*} mL_k=\sum \limits_{j=-(2u+1)}^{2u+v}\varepsilon_j(m)F_{k+j}
\end{align*} 
and $[-(2u+1),2u+v] \subseteq \left[k-\ell-1,k+\ell+1\right]$ where $\ell$ is such that $F_{\ell} \leq m < F_{\ell +1}$. 

\end{lem}

The last statement of \Cref{Lucas} has an important role for the construction of sums with non-interfering terms. Indeed, for an integer $m$ and $k$ large enough (of order of length of $m$), the expansion of $mL_k$ is the expansion of $m$ centered around $F_k$. Notice that this result is different from the $q$-base since the digits here appear on both sides of the expansion of $m$. We will see the impact in our main result \Cref{Main_theorem_polynomial} with the occurence of $N^{1/(2d)}$ in the place of $N^{1/d}$ that we got in the case of $q$-base expansion, see~\cite{popoli2020}.

\bigskip

Let us start with the case $P(X)=X^d$ for $d\geq 2$. In the following we write $\overline{\lambda}=\lambda \mod 2$ with $\overline{\lambda}\in \{0,1\}$. This case is the building brick for the general case of $P(X) \in \mathbb{Z}[X]$ with $P(\mathbb{N}_0) \subset \mathbb{N}_0$.

To begin with, we study the expression of $(n+L_{\ell})^d$ in terms of a sum of distinct Lucas numbers. 

\begin{lem}\label{decomposition}
Let $n\geq 0$ be an integer and $\ell\geq 0$ be an even integer. We have \begin{align*}
(n+L_{\ell})^d&=\sum \limits_{i=0} ^{\frac{d-\overline{d}}{2}} \eta_{i,0}n^{d-2i}+\sum \limits_{\lambda=1}^d\left(\sum \limits_{i=\frac{\lambda-\overline{\lambda}}{2}}^{\frac{d-\overline{d}}{2}-\overline{\lambda}}\eta_{i,\lambda} n^{d-2i-\overline{\lambda}} \right) L_{\lambda \ell},
\end{align*} with $\eta_{i,\lambda}=\binom{d}{2i+\overline{\lambda}}\binom{2i+\overline{\lambda}}{i-\frac{\lambda-\overline{\lambda}}{2}}$.
\end{lem}

\begin{proof}
Let $d$ be even. Then by \Cref{Lucas} we have \begin{align*}
(n+L_{\ell})^d&=\sum \limits_{i=0} ^{d/2} \binom{d}{2i} n^{d-2i} L_{\ell}^{2i}+\sum \limits_{i=0}^{d/2-1}\binom{d}{2i+1}n^{d-2i-1}L_{\ell}^{2i+1}, \\&=n^{d}+ \sum \limits_{i=1} ^{d/2} \binom{d}{2i} n^{d-2i} \left(\sum \limits_{j=0} ^{i-1} \binom{2i}{j} L_{(2i-2j)\ell}+\binom{2i}{i} \right) \\  &\qquad+\sum \limits_{i=0}^{d/2-1}\binom{d}{2i+1}n^{d-2i-1} \left(\sum \limits_{j=0} ^{i} \binom{2i+1}{j} L_{(2i-2j+1)\ell} \right),  \\ &=\sum \limits_{i=0} ^{\frac{d}{2}} \binom{d}{2i}\binom{2i}{i}n^{d-2i}\\ &\qquad+ \sum \limits_{\lambda=1}^d\left(\sum \limits_{i=\frac{\lambda-\overline{\lambda}}{2}}^{\frac{d}{2}-\overline{\lambda}}\binom{d}{2i+\overline{\lambda}}\binom{2i+\overline{\lambda}}{i-\frac{\lambda-\overline{\lambda}}{2}} n^{d-2i-\overline{\lambda}} \right) L_{\lambda \ell}.
\end{align*}
For $d$ odd the proof runs along the same lines and we have \begin{align*}
(n+L_{\ell})^d=\sum \limits_{i=0} ^{\frac{d-1}{2}} &\binom{d}{2i}\binom{2i}{i}n^{d-2i}\\ &\qquad+ \sum \limits_{\lambda=1}^d\left(\sum \limits_{i=\frac{\lambda-\overline{\lambda}}{2}}^{\frac{d-1}{2}-\overline{\lambda}}\binom{d}{2i+\overline{\lambda}}\binom{2i+\overline{\lambda}}{i-\frac{\lambda-\overline{\lambda}}{2}} n^{d-2i-\overline{\lambda}} \right) L_{\lambda \ell}.
\end{align*}
\end{proof}
\begin{lem} \label{interference}
Let $n\geq d$ be an integer, $d\geq 3$ an odd integer and $1\leq \lambda \leq d$. Then we have \begin{align*} \sum \limits_{i=\frac{\lambda-\overline{\lambda}}{2}}^{\frac{d-1}{2}-\overline{\lambda}}\binom{d}{2i+\overline{\lambda}}\binom{2i+\overline{\lambda}}{i-\frac{\lambda-\overline{\lambda}}{2}} n^{d-2i-\overline{\lambda}}<\sum \limits_{i=0} ^{\frac{d-1}{2}} \binom{d}{2i}\binom{2i}{i}n^{d-2i}.
\end{align*}
\end{lem}

\begin{proof}
We distinguish the two following cases:\begin{enumerate}
\item If $\overline{\lambda}=0$, we have \begin{align*}
\sum \limits_{i=0} ^{\frac{d-1}{2}} \binom{d}{2i}\binom{2i}{i}n^{d-2i}=\sum \limits_{i=0} ^{\frac{\lambda}{2}-1} \binom{d}{2i}\binom{2i}{i}n^{d-2i}+\sum \limits_{i=\frac{\lambda}{2}} ^{\frac{d-1}{2}} \binom{d}{2i}\binom{2i}{i}n^{d-2i}.
\end{align*}Since $\binom{2i}{i}\geq \binom{2i}{i-\frac{\lambda}{2}}$ for all $i \geq \frac{\lambda}{2}$, we have the result.
\item If $\overline{\lambda}=1$, we have \begin{align*}
\sum \limits_{i=0} ^{\frac{d-1}{2}} \binom{d}{2i}\binom{2i}{i}n^{d-2i}&=\sum \limits_{i=0} ^{\frac{\lambda-1}{2}-1} \binom{d}{2i}\binom{2i}{i}n^{d-2i}\\ &\qquad +\sum \limits_{i=\frac{\lambda-1}{2}} ^{\frac{d-1}{2}-1} \binom{d}{2i}\binom{2i}{i}n^{d-2i}+d\binom{d-1}{\frac{d-1}{2}}n,
\end{align*}
and \begin{align*}
\sum \limits_{i=\frac{\lambda-1}{2}}^{\frac{d-1}{2}-1}&\binom{d}{2i+1}\binom{2i+1}{i-\frac{\lambda-1}{2}} n^{d-2i-1}\\ &\qquad \leq \frac{d-(\lambda-1)}{n} \sum \limits_{i=\frac{\lambda-1}{2}}^{\frac{d-1}{2}-1}\binom{d}{2i}\binom{2i}{i-\frac{\lambda-1}{2}} n^{d-2i}.
\end{align*} Again we get the result for $n\geq d$. 
\end{enumerate}
\end{proof}

\begin{remark}
The condition $n\geq d$ in \Cref{interference} has no impact on the quality of the lower bound of the complexity since it does not depend on $\ell$.
\end{remark}

Now, for $k\geq 1$ consider \begin{align*}
t(k)=m_3L_{6k}-m_2L_{4k}+m_1L_{2k}+m_0L_0,
\end{align*}
with real parameters $m_i$, see~\cite{stoll2013}. For $d\geq 1$ set \begin{align*}
T_d(k)=t(k)^d=\sum \limits_{0\leq i \leq 3d}c_iL_{2ik}.
\end{align*}
Note that for $k$ sufficiently large, the coefficients $c_i$ are independent from $k$. 

\medskip

We will need the following auxiliary result~\cite[Lemma 4]{stoll2013}.
\begin{lem}[{\cite{stoll2013}}]
Let $M \geq 1$, and $m_0,m_1,m_2,m_3 \in \mathbb{R}$ with \begin{align*}
1\leq m_0,m_1,m_3<M, \\ 0<m_2<\frac{1}{d^3(32M)^d}.
\end{align*}
Then we have $c_{3d}>0$, $c_{3d-1}<0$ and $c_i>0$ for $i=0,1,\ldots,3\ell -2$.
\end{lem}
This lemma will be useful to construct two elements with different sums of digits in Zeckendorf base when we take their $d$th powers. Indeed, only the subdominant coefficient is negative so we are able to create a block of digits $\tL\tO\tL\tO\cdots\tL\tO$ of length $k$ for any $k$ large enough. Indeed, as we will state later, the transition from $k$ to $k+1$ adds exactly one block of $\tL\tO$. We have already used a similar method in the usual $2$-base, see \cite[proof of Lemma 6]{popoli2020}. 
 
In the following, let $\alpha \geq 1$ be an integer such that \begin{align*}
\varphi^{\alpha}>d^3\varphi(32\varphi)^d,
\end{align*}
and $m_2=1$ and $m_0,m_1,m_3$ integers such that \begin{align} \label{m0m1m3}
\varphi^{\alpha-1}\leq m_0,m_1,m_3<\varphi^{\alpha}.
\end{align}

Under these conditions, $T_d(k)$ and $T_d(k+1)$ have all positive integral coefficients with the only exception of the coefficient of $L_{2k(3d-1)}$ and $L_{2(k+1)(3d-1)}$ respectively, see~\cite{stoll2013}. Notice that these coefficients are the same for these two polynomials for $k$ large enough. 

\begin{lem} \label{key_monomial}
Let $d\geq 3$ be odd and $k\geq 0$ be a sufficiently large integer. Let $n$ be an integer such that $n\geq d$ and \begin{align} \label{3k} \sum \limits_{i=0} ^{(d-1)/2} \binom{d}{2i}\binom{2i}{i}n^{d-2i} < F_{3k}.
\end{align} Then we have 
\begin{align} s_Z((n+L_{6k+2})^d) &= s_Z((n+L_{6k+4})^d), \label{equality} \\ s_Z(T_d(k)) \bmod 2&\neq s_Z(T_d(k+1)) \bmod 2. \label{diff}
\end{align}
\end{lem}

\begin{proof}

Let $\ell\geq 0$ be an even integer, we have already proved in \Cref{decomposition} that for odd $d$ we have  \begin{align*}
(n+L_{\ell})^d&=\sum \limits_{i=0} ^{\frac{d-1}{2}} \eta_{i,0}n^{d-2i}+\sum \limits_{\lambda=1}^d\left(\sum \limits_{i=\frac{\lambda-\overline{\lambda}}{2}}^{\frac{d-1}{2}-\overline{\lambda}}\eta_{i,\lambda} n^{d-2i-\overline{\lambda}} \right) L_{\lambda \ell}.
\end{align*}

Condition \eqref{3k} and \Cref{interference} imply that all coefficients in front of each $L_{\lambda \ell}$, for $1\leq \lambda \leq d$, are $<F_{3k}$. Since the aim is to build a sum with non-interfering terms, we study for the range of digits of each term in the sum with the help of \Cref{Lucas}. Hence for sufficiently large $\ell$, the range of digits of \begin{align*} \left(\sum \limits_{i=\frac{\lambda-\overline{\lambda}}{2}}^{d/2-\overline{\lambda}}\binom{d}{2i+\overline{\lambda}}\binom{2i+\overline{\lambda}}{i-\frac{\lambda-\overline{\lambda}}{2}} n^{d-2i-\overline{\lambda}} \right) L_{\lambda \ell}
\end{align*} is included in the interval $\left[\lambda \ell - 3k,\lambda \ell + 3k\right]$ for $\lambda > 0$ and $\left[0,3k-1\right]$ for $\lambda=0$. If we suppose that these intervals are disjoints plus one small gap, see \Cref{remark}, we have a non-interfering sum, i.e when we suppose \begin{align*}
[0,3k-1]& \cap [ \ell -3k-1,\ell +3k+1]=\emptyset, \\
[ \ell -3k-1,\ell +3k+1] &\cap [2\ell -3k-1,2\ell +3k+1]=\emptyset, \\
&\ldots, \\
[(d-1)\ell-3k-1,(d-1)\ell+3k+1]&\cap [d\ell -3k-1,d\ell +3k+1]=\emptyset.
\end{align*} For this to happen, it is sufficient to suppose $\ell>6k$ and we have imposed $\ell$ even. So we can choose $\ell=6k+2$ in order to have a non-interfering sum. An identical proof works for $\ell=6k+4$ and \eqref{equality} is proved.

The second part follows directly from \cite[Lemma 3]{stoll2013}. For $m_1,m_2\geq 1$ and $k_1>k_2$ large enough, we have \begin{align*}
s_Z(m_1L_{2k_1}-m_2L_{2k_2})=k_1-k_2+C(m_1,m_2),
\end{align*}where $C(m_1,m_2)$ only depends on $m_1$ and $m_2$.
For $k$ sufficiently large we therefore have \begin{align*}s_Z(t(k)^d)&=s_Z(c_{3d}L_{6kd}-(-c_{3d-1})L_{2k(3d-1)}+\cdots+c_0L_0) \\ &=(3dk-k(3d-1))+C(c_{3d},c_{3d-1})+\kappa \\ &=k+C(c_{3d},c_{3d-1})+\kappa,
\end{align*}
for some $\kappa$ independent from $k$ and \begin{align*} s_Z(t(k+1)^d)&=s_Z(c_{3d}L_{6(k+1)d}-(-c_{3d-1})L_{2(k+1)(3d-1)}+\cdots+c_0L_0) \\ &=(3d(k+1)-(k+1)(3d-1))+C(c_{3d},c_{3d-1})+\kappa \\ &=k+1+C(c_{3d},c_{3d-1})+\kappa.
\end{align*}
This proves \eqref{diff}.

\end{proof}
We are now able to prove \Cref{Main_theorem_polynomial} in the particular case $P(X)=X^d$.
\begin{proof}[Proof of \Cref{Main_theorem_polynomial}]
We suppose the same hypotheses as in \Cref{key_monomial} and we choose $k\geq 0$ such that \begin{align}\label{choice_of_k}
t(k+1)<N\leq t(k+2).
\end{align}
Thus condition \eqref{3k} gives $n^d \ll F_{3k}$ and $n \ll F_k^{3/d}$. As we have already stated in \Cref{Method_Jansen}, we have \begin{align*}
M(\mathcal{S}_{Z,P},N) \gg F_k^{3/d}
\end{align*} since we have two blocks of length $\ll F_k^{3/d}$ with two different successors. It remains to ensure that \begin{align}\label{not_first_block}
t(k) \gg F_{k}^{3/d}+L_{6k+2},
\end{align} since we need $t(k)^d$ a successor of the non-interfering block. We have $t(k)\sim m_3L_{6k}\sim m_3\varphi^{6k}$ and $F_{k}^{3/d}+L_{6k+2}\sim L_{6k+2} \sim \varphi^{6k+2}$ as $k\to +\infty$. Since by \eqref{m0m1m3}, $m_3\geq d^3(32\varphi)^d$, we have $t(k)\gg 32^dd^3\varphi^{6k+d}$ for $k\to +\infty$. Therefore for sufficiently large $k$, \eqref{not_first_block} is satisfied. For the same reasons, we need \begin{align} \label{not_first_block_bis}
t(k+1) \gg F_{k}^{3/d}+L_{6k+4}.
\end{align}
A very similar proof of \eqref{not_first_block} gives \eqref{not_first_block_bis} for $k$ sufficiently large. 
Furthermore, \eqref{choice_of_k} gives \begin{align*}
N\leq t(k+2) \ll L_{6(k+2)} \ll F_{k}^6.
\end{align*} 
We finally get \begin{align*}
M(\mathcal{S}_{Z,P},N) \gg F_k^{3/d} \gg N^{1/(2d)},
\end{align*} 
and the theorem is proved if $d$ is odd. 

For $d$ even, we prove a similar result as \Cref{key_monomial} and a similar proof works for the same reasons (we omit the details). Thus \Cref{Main_theorem_polynomial} is proved in the case $P(X)=X^d$. 
\end{proof}

\subsection{Polynomial subsequences} We consider now the general case with $P(X)=\alpha_dX^d+\cdots+\alpha_0$ a polynomial such that 
$P(\mathbb{N}_0) \subset \mathbb{N}_0$ and $\alpha_d=1$. Such as before, we shall determine exactly $P(n+L_k)$ for the integers $n\geq 0$ et $k\geq 0$. 

\begin{lem}\label{decomposition2}
We have for integers $n\geq 0$ et $k\geq 0$, \begin{align*}
P(n+L_k)=\sum \limits_{0 \leq \lambda \leq d}\beta_{\lambda}(n)L_{\lambda k}
\end{align*} with $\beta_{\lambda}(n)=\sum \limits_{\lambda \leq i \leq d}\alpha_i \sum \limits_{j=\frac{\lambda-\overline{\lambda}}{2}}^{\lfloor \frac{i-\overline{\lambda}}{2} \rfloor }\gamma_{i,j,\lambda}n^{i-2j-\overline{\lambda}}$ and $ \gamma_{i,j,\lambda}=\binom{i}{2j+\overline{\lambda}}\binom{2j+\overline{\lambda}}{j-\frac{\lambda-\overline{\lambda}}{2}}$.
\end{lem}

\begin{proof}
The proof is similar to \Cref{decomposition}.
\end{proof}

\begin{lem} \label{interference2} Let $n\geq C_d$ and $\lambda \geq 1$ be integers with $C_d$ an absolute constant depending only on $d$. Then we have $\beta_{\lambda}(n)<\beta_0(n)$.
\end{lem}

\begin{proof}
The proof is again similar to \Cref{interference}.
\end{proof}

Let $\mu$ be an integer. We have by \Cref{decomposition2} \begin{align} \label{Polysum}
P(n+L_{2d\mu k+2})=\beta_0(n)+\beta_1(n)L_{2d\mu k+2}+\cdots+\beta_d(n)L_{d(2d\mu k+2)}.
\end{align}

\begin{remark}
We introduce $\mu$ since it is sufficient to adjust the non-interfering block to have the same result. Later, we shall take $\mu$ depending only on $d$. \end{remark}

To prove a lower bound on the \textit{Nth maximum order complexity} of $\mathcal{S}_{Z,P}$, we look for an analogue of \Cref{key_monomial}. The analogue of the first part of this lemma is described as follows.

\begin{lem} \label{key_polynomial}
Let $k\geq 0$ be a sufficiently large integer and $r>1$. For any integer $C_d<n<F_{\mu k}$ we have \begin{align}
s_Z(P(n+L_{2d\mu k+2}))=s_Z(P(n+L_{2d\mu k+2r})).
\end{align}
\end{lem}
\begin{proof}
For $n<F_{\mu k}$ we have $\beta_0(n) \ll L_{d\mu k}$. Then the terms in \eqref{Polysum} are non-interfering for $k$ large enough such as in the proof of \Cref{key_monomial} where \Cref{interference} is replaced by \Cref{interference2}. Thus we have  \begin{align*}
s_Z(P(n+L_{2d\mu k}))=s_Z(\beta_0(n))+s_Z(\beta_1(n))+\cdots+s_Z(\beta_d(n)).
\end{align*}
Under the same hypothesis, we have for any $r>1$, \begin{align*}
s_Z((P(n+L_{2d\mu k+2r}))=s_Z(\beta_0(n))+s_Z(\beta_1(n))+\cdots+s_Z(\beta_d(n)).
\end{align*}
Therefore, for any $n<F_{\mu k}$, $k$ large enough, and $r>1$, we have \begin{align*}
s_Z(P(n+L_{2d\mu k}))=s_Z((P(n+L_{2d\mu k+2r})).
\end{align*} 
Thus the lemma is proved.
\end{proof}

In order to obtain the analogue of the second part of \Cref{key_monomial}, we are now looking for an integer $n$ of the form $L_{\lambda}$ for some $\lambda$ and $r>1$ such that \begin{align} \label{lambda}
s_Z(P(L_{\lambda}+L_{2d \mu k+2})) \neq s_Z(P(L_{\lambda}+L_{2d \mu k+2r})).
\end{align}
According to \Cref{key_polynomial}, we need $\lambda \geq \mu k$. Note that we want to prove that $M(\mathcal{S}_{Z,P},N)\gg N^{1/(2d)}$ in order to have the same bound as the one for the monomial case. The following result gives sufficient conditions for the size of $\lambda$ and $r$ for this bound.

\begin{lem}
If $\lambda\leq 2d\mu k$ and $r$ is constant, then we have $M(\mathcal{S}_{Z,P},N)\gg N^{1/(2d)}$.
\end{lem}

\begin{proof}
If we have $\lambda$ such that \eqref{lambda} is verified, we show in the same way as before that $$M(\mathcal{S}_{Z,P},N) \geq F_{\mu k},$$ where $k$ is chosen in a way that \begin{align*}
L_{\lambda}+L_{2d\mu k+2r}<N\leq L_{\lambda}+L_{2d\mu (k+1)+2r}.
\end{align*}
This implies \begin{align*}
N &\ll L_{\lambda}+L_{2d\mu k+2r}\\
& \ll F_{\mu k}^{ \lambda/\mu k}+F_{\mu k}^{2d+2r/ \mu k}
\end{align*}
If $\lambda \leq 2d\mu k$ and $r$ is constant, we have $N\ll F_{\mu k}^{2d}$ and $M(\mathcal{S}_{Z,P},N)\gg N^{1/(2d)}$. Notice that $\mu$ independent from $k$ is also required to have this result. 
\end{proof}

We are now looking for $(\lambda,r)$ such that $\mu k \leq \lambda \leq 2d\mu k$, $r$ constant and \eqref{lambda} is verified. 

As we will state later, only two Lucas numbers will interfere. The following lemma describes all the possibilities for this interference. 
\begin{lem}\label{two_Lucas_numbers}
Let $k,\ell$ be two integers such that $k\geq \ell \geq 5$. We have \begin{align*}
s_Z(L_k+L_{\ell}) \equiv 1 \pmod 2 \quad \Longleftrightarrow \quad k=\ell +2.
\end{align*}
\end{lem}
\begin{proof}
We have $L_k=F_{k+1}+F_{k-1}$ and $L_{\ell}=F_{\ell+1}+F_{\ell-1}$. So we shift two blocks of $\tL\tO\tL$ one against the other and look for transversals in the carry propagations. Note that if $k\geq \ell +4$, there is no carry propagation at all and $s_Z(L_k+L_{\ell})=4 \equiv 0 \pmod 2$. By symmetry, there remain only four cases to consider. We suppose $\ell \geq 5$: \begin{itemize}
\item Case $k=\ell$:  \[\begin{tabular}{cccccccccll}
& & & $\tL$  & $\tO$ & $\tL$\\
+&   & & $\tL$ & $\tO$ & $\tL$ \\
\hline
=&&& $\tD$ & $\tO$ & $\tD$ &  \\
\hline 
=&& $\tL$ & $\tO$ & $\tL$ & $\tL$ & $\tO$ & $\tL$\\ 
\hline
=&$\tL$ & $\tO$ & $\tO$ & $\tO$ & $\tO$ & $\tO$  & $\tL$ \\
\end{tabular}\]Then we have $s_Z(L_k+L_{\ell})\equiv 0 \pmod 2$. 
\item Case $k=\ell +1$:   \[\begin{tabular}{cccccccccll}
& & $\tL$  & $\tO$ & $\tL$\\
+&  &  & $\tL$ & $\tO$ & $\tL$ \\
\hline
=&$\tL$ & $\tO$ & $\tL$ & $\tO$ & $\tO$ & &  \\
\end{tabular}\]Then we have $s_Z(L_k+L_{\ell})\equiv 0 \pmod 2$. 
\item Case $k=\ell +2$:  \[\begin{tabular}{cccccccccll}
& & $\tL$  & $\tO$ & $\tL$\\
+& & &  & $\tL$ & $\tO$ & $\tL$ \\
\hline
=&& $\tL$ & $\tL$ & $\tO$ & $\tO$ & $\tD$ \\
\hline 
=&$\tL$ & $\tO$ & $\tO$ & $\tO$ & $\tL$ & $\tO$ & $\tO$ & $\tL$ \\
\end{tabular}\]Then we have $s_Z(L_k+L_{\ell})\equiv 1 \pmod 2$. 
\item Case $k=\ell +3$:  \[\begin{tabular}{cccccccccll}
& & $\tL$  & $\tO$ & $\tL$\\
+& & & &  & $\tL$ & $\tO$ & $\tL$ \\
\hline
=&& $\tL$ & $\tL$ & $\tO$ & $\tO$ & $\tO$ & $\tL$ \\
\hline 
=&$\tL$ & $\tO$ & $\tO$ & $\tO$ & $\tO$ &   $\tO$ & $\tL$ \\
\end{tabular}\]Then we have $s_Z(L_k+L_{\ell})\equiv 0 \pmod 2$. 
\end{itemize}
Thus the lemma is proved. 
\end{proof}

\begin{proof}[Proof of \Cref{Main_theorem_polynomial}]
We now investigate the interferences in 
\begin{align*} 
P(L_{\lambda}+L_{2d\mu k+2})=\beta_0(\lambda)+\beta_1(\lambda)L_{2d\mu k+2}+\cdots+\beta_d(\lambda)L_{d(2d\mu k+2)}.
\end{align*}
For each $\beta_i(\lambda)L_{i(2d\mu k+2)}$ we locate the least significant digit and the most significant digit. Then we have the following possible interferences \begin{align*}
L_{\lambda d}&+(-1)^{(d-1)\lambda}L_{2d\mu k+2 -(d-1)\lambda} \\
L_{2d\mu k+2 +(d-1)\lambda}&+(-1)^{(d-2)\lambda}L_{2(2d\mu k+2) -(d-2)\lambda} \\ &\cdots \\ L_{(d-1)(2d\mu k+2)+\lambda}&+L_{d(2d\mu k+2)}
\end{align*} since $\alpha_d=1$.

We will now choose $\lambda$ even. Thus we have for the first interference $L_{\lambda d}+L_{2d\mu k -(d-1)\lambda+2}$. By \Cref{two_Lucas_numbers} we have for sufficiently large $k$ \begin{align}
s_Z(L_{\lambda d}+L_{2d\mu k -(d-1)\lambda+2})=1 &\Leftrightarrow 2d\mu k-(d-1)\lambda+2=\lambda d +2 \nonumber \\&\Leftrightarrow \lambda=\frac{2d\mu k}{2d-1}. \label{final mu}
\end{align}

Hence it is sufficient to take $\mu=2d-1$ and $\lambda=2dk$. Furthermore, all other possible interferences are non-existing by this choice of $\lambda$ and $\mu$. We finally need to check if $\lambda$ satisfies all the conditions: $\lambda$ is even and $(2d-1)k \leq \lambda \leq 2d(2d-1)k$. Furthermore we can take $r=2$ for example and all the appearing blocks do not interfere. Summing up, we have proved that $s_Z\left(P\left(L_{\lambda}+L_{2d(2d-1) k +2}\right)\right)\neq s_Z\left(P\left(L_{\lambda}+L_{2d(2d-1) k +4}\right)\right)$. Thus, we have pro\-ved the general case of \Cref{Main_theorem_polynomial}.
\end{proof}

\begin{remark}
The choice of $L_{2d\mu k+2}$ in the proof is motivated to have an integer solution for $\lambda$ in \eqref{final mu}. Indeed, if we had chosen $L_{2d\mu k}$ instead, \eqref{final mu} would not have provided an integer for all $k$. 
\end{remark}

\section{Final remarks}\label{sec_final_remarks}

\subsection{Generalizations}
Our result can be generalized to other numeration systems without new methods. Let us first comment on Ostrowski's $\alpha$-numeration system, see \cite[Theorem 3.9.1]{alloucheshallit2003}, related to the continued fraction of an irrational number. 

\begin{Def}[Continued fraction]
Let $\alpha$ be a real number, we define the sequence $(a_i)_i$ such that \[\alpha=a_0+\frac{1}{a_1+\frac{1}{a_2+\cdots}}=\left[a_0,a_1,\ldots\right]
\]
and we define $p_{-2}=0$, $p_{-1}=1$, $q_{-2}=1$, $q_{-1}=0$,  and $p_n=a_np_{n-1}+p_{n-2}$, $q_n=a_nq_{n-1}+q_{n-2}$. Then we have $\frac{p_n}{q_n}=\left[a_0,a_1,\ldots,a_n\right]$. The sequence $(q_n)_n$ is called denominators of the convergents of the continued fraction of $\alpha$. 
\end{Def}

\begin{Def}[Ostrowski's $\alpha$-numeration system] 
Let $\alpha$ be a non rational real number, and let $(q_n)$ be the sequence of the denominators of the convergents of the continued fraction of $\alpha$. Then every non-negative integer $n$ can be represented uniquely in the form \begin{align*}
n=\sum \limits_{0\leq i \leq j}\varepsilon_i q_i
\end{align*}
where the $\varepsilon_i$ are integers satisfying the following conditions: \begin{enumerate}
\item $0\leq b_0<a_1$
\item $0\leq b_i \leq a_{i+1}$, for $i\geq 1$. 
\item For $i\geq 1$, if $b_i=a_{i+1}$ then $b_{i-1}=0$.
\end{enumerate}
\end{Def}

For the specific case $\alpha=\varphi$, we have the Zeckendorf expansion since $\varphi=[1,1,1,\ldots]$. 

Our result might be generalized to a quadratic irrational $\alpha$ such that the continued fraction of $\alpha$ is of the form $\alpha=[1,a,a,\ldots]$. In this particular case, we have $q_{n+2}=\frac{1}{a}q_{n+1}+q_n$. We denote by $\gamma$ the zero of the polynomial $x^2-\frac{1}{a}x-1$ such that its Galois conjugate $\overline{\gamma}$ verifies $|\overline{\gamma}|<|\gamma|$. Thus we have $q_n=\frac{1}{\sqrt{a^{-2}+4}}(\gamma^n-\overline{\gamma}^n)$, an analogous formula that for Fibonacci numbers. We define an analogue of Lucas numbers by $L'_n=\gamma^n+\overline{\gamma}^n$ for all $n\geq 0$. Then, in this particular case, our new notion of ``Lucas numbers'' verifies $L'_n=q_{n+1}+q_{n-1}$ and an analogous formula of \Cref{Lucas} holds true. So it might be possible to generalize our result in this case since no more input in the proof is needed. 

\subsection{Conjectures}\label{secConjecture}
We can compute the maximum order complexity thanks to Blumer's DAWG (Direct Acyclic Weighted Graph) algorithm (see~\cite{blumerblumerhausslerehrenfeuchtchenseiferas1985} and~\cite{jansen1989}). The C++ program that computes the maximum order complexity of a sequence is available on the web page of the second author.\footnote{\url{https://iecl.univ-lorraine.fr/membre-iecl/popoli-pierre/}} By using this program, we are able to formulate some heuristically supported conjectures. 

\begin{figure}[H]
\centering
\includegraphics[width=0.75\textwidth]{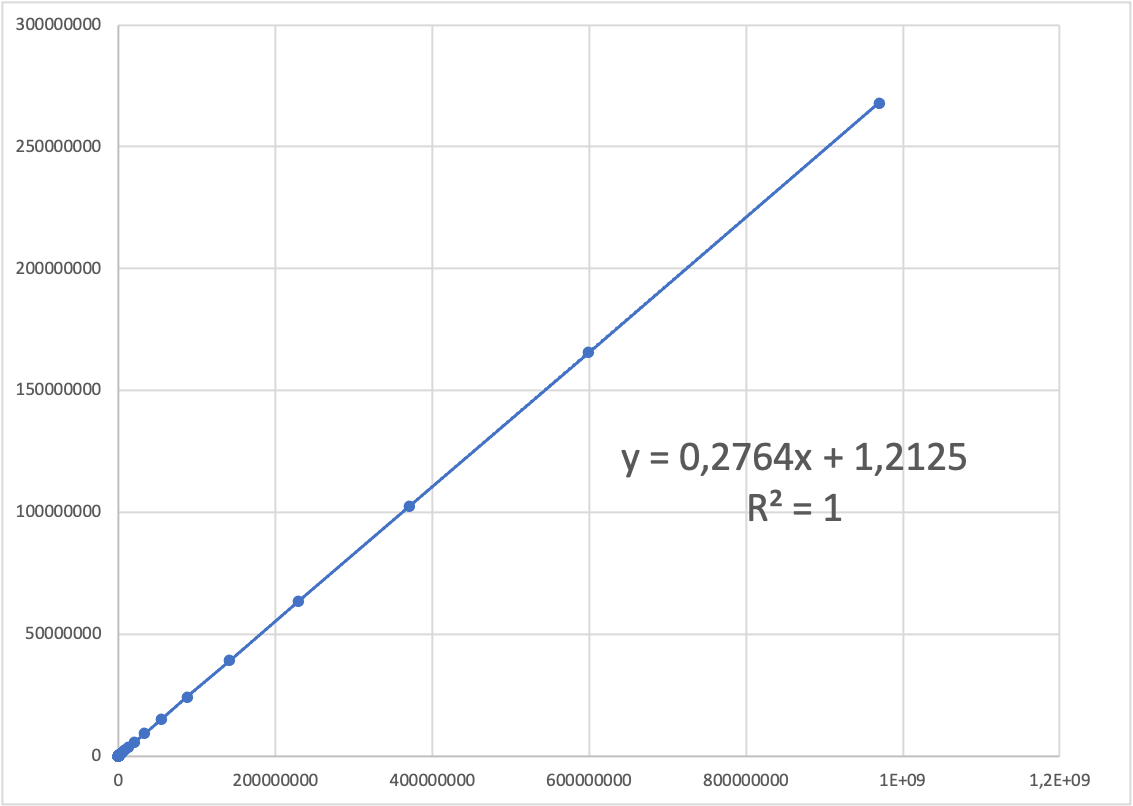}
\caption{Maximum order complexity for the sum of digits function in Zeckendorf base.}
\label{Zeckendorf_linear}
\end{figure}

\begin{conjecture}
We conjecture that $M(\mathcal{S}_Z,N) \sim \frac{1}{1+\varphi^2}N$. \Cref{Zeckendorf_linear} indicates that the maximum order complexity for $\mathcal{S}_Z$ is linear with coefficient $\frac{1}{1+\varphi^2}=0.27639\ldots$
\end{conjecture}

Let us compare the maximum order complexity along polynomial subsequences of the Thue--Morse sequence to the Zeckendorf expansion sequence.

The plots for squares and cubes (and the large values of $R^2$), see Figure 2 and Figure 3, indicate that the growth is close to $x^{1/2}$ and $x^{1/3}$, respectively. We therefore formulate the following conjecture :
\begin{figure}[H]
\centering
\includegraphics[width=0.75\textwidth]{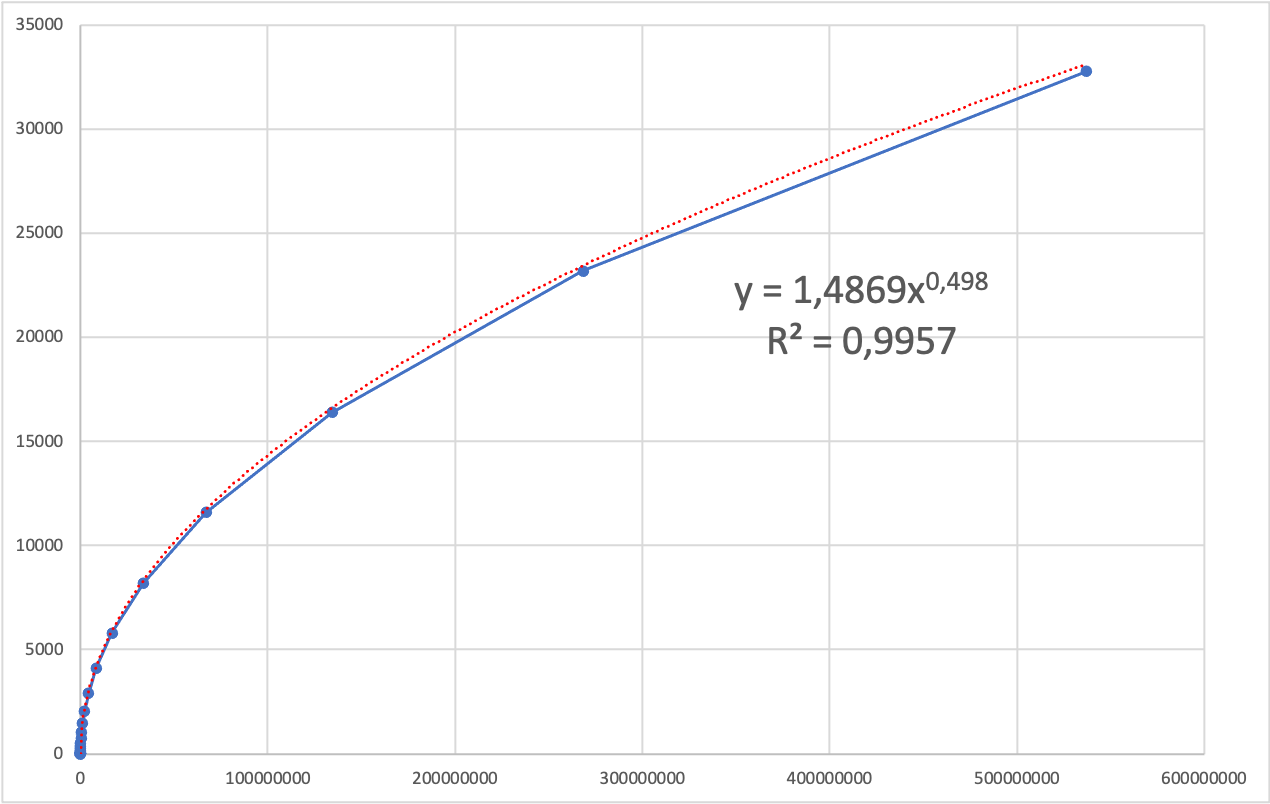}
\caption{Maximum order complexity for the Thue--Morse sequence along squares.}
\label{Thue-Morse_squares}
\end{figure}

\begin{figure}[H]
\centering
\includegraphics[width=0.75\textwidth]{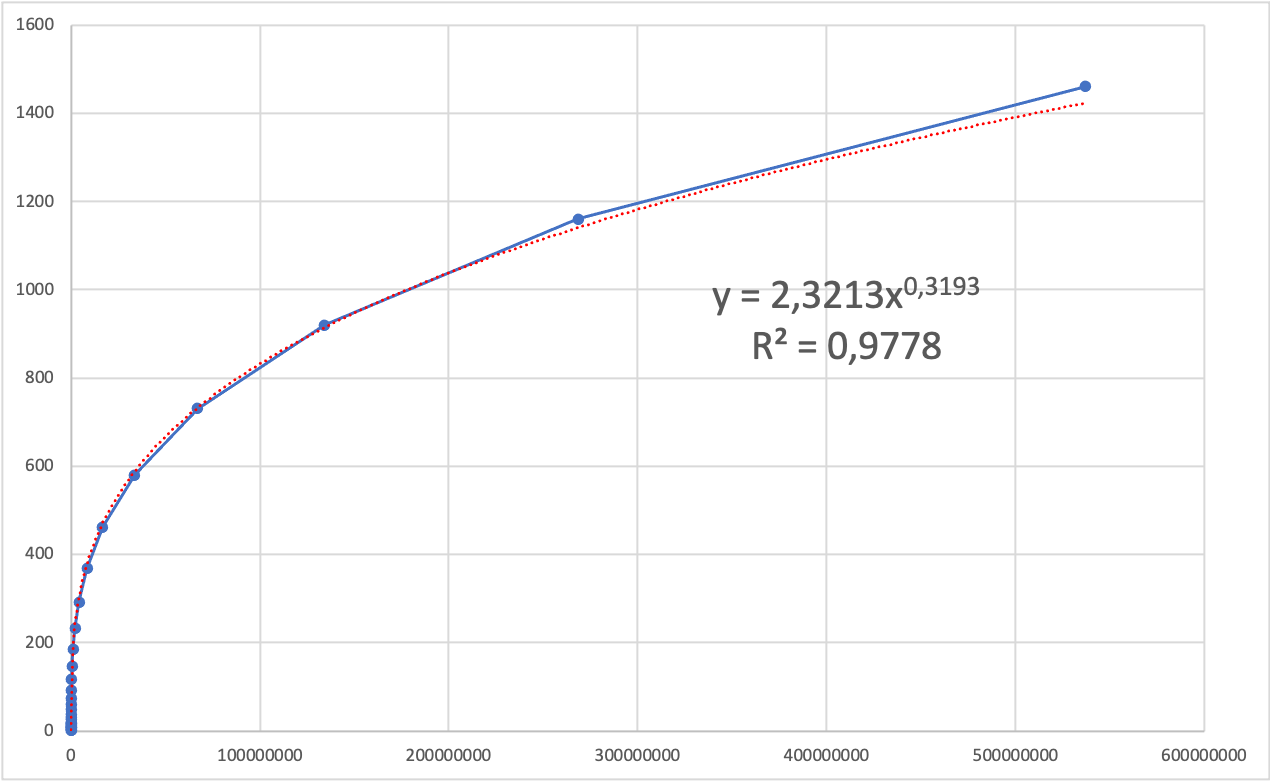}
\caption{Maximum order complexity for the Thue--Morse sequence along cubes.}
\label{Thue-Morse_cubes}
\end{figure}

\begin{conjecture} \label{conjecture_TM_Poly}
The Thue--Morse sequence along polynomial subsequences, denoted by $\mathcal{T}_P$ for a polynomial $P$ of degree $d$, verifies $M(\mathcal{T}_P,N)\asymp N^{1/d}$, i.e. there are $c,C>0$ such as for all $N$ large enough we have \begin{align*}
cN^{1/d}\leq M(\mathcal{T}_P,N) \leq CN^{1/d}.
\end{align*}
\end{conjecture}

A proof of this conjecture would imply that the lower bound proved by Popoli \cite{popoli2020} is optimal.
\begin{minipage}[b]{0.45\linewidth}\centering

\end{minipage}
\begin{figure}[H]
\centering
\includegraphics[width=0.75\textwidth]{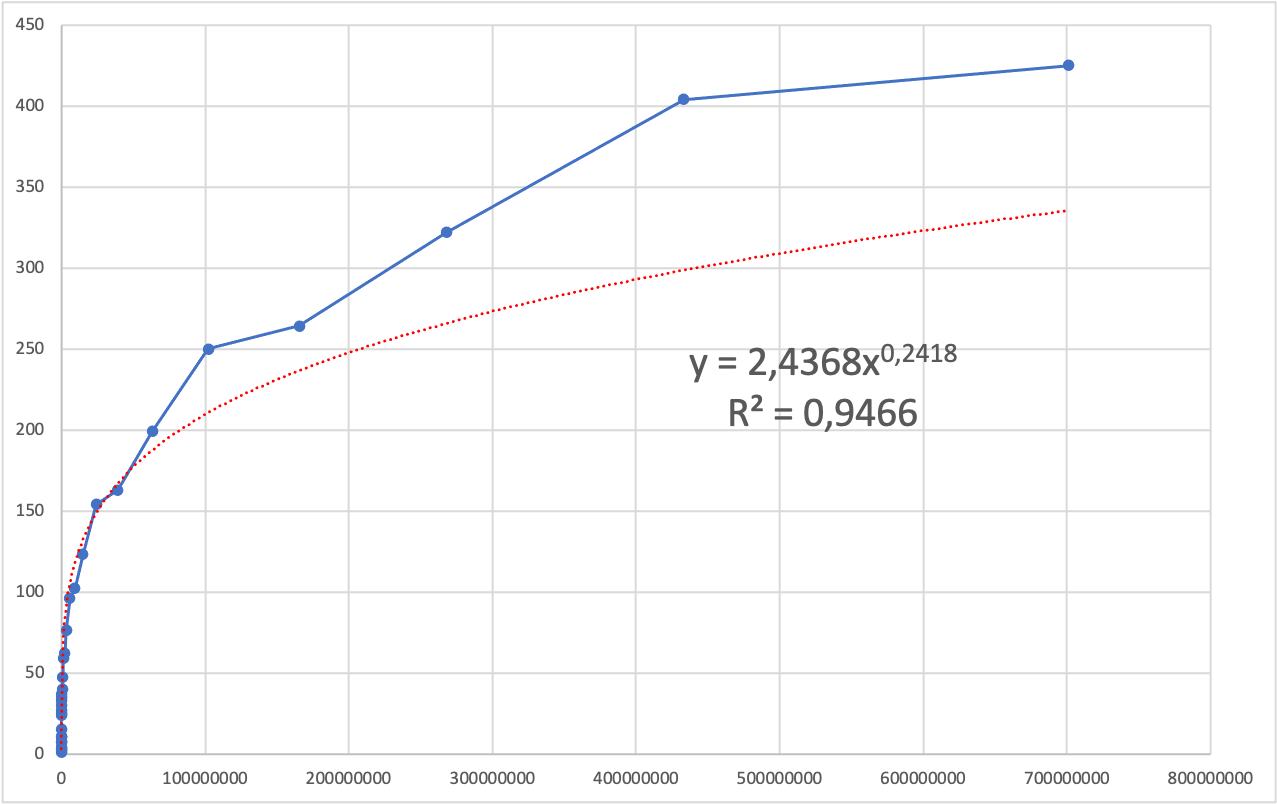}
\caption{Maximum order complexity of $\mathcal{S}_Z$  along squares.}
\label{Zeckendorf_squares}
\end{figure}

The maximum order complexity of $\mathcal{S}_\varphi$ is algorithmically more difficult to handle. Motivated by the results on Thue--Morse we conjecture the following:
\begin{conjecture} \label{Conjecture_Zeckendorf_Poly}
The sequence $\mathcal{S}_Z$ along polynomial subsequences, denoted by $\mathcal{S}_{Z,P}$ for a polynomial $P$ of degree $d\geq 2$, verifies $M(\mathcal{S}_{Z,P},N)\asymp N^{1/(2d)}$, i.e. there are $c,C>0$ such as for all $N$ large enough we have \begin{align*}
cN^{1/(2d)}\leq M(\mathcal{S}_{Z,P},N) \leq CN^{1/(2d)}.
\end{align*}
\end{conjecture}

This is, to some extent, supported by the plot for squares, see \Cref{Zeckendorf_squares}. Again, a proof of this conjecture would imply that the lower bound in \Cref{Main_theorem_polynomial} is sharp. Notice that there is a a larger gap for the maximum order complexity between the linear case and the quadratic case for $\mathcal{S}_Z$ compared to the classical Thue--Morse sequence. 

\medskip

We conclude the discussion by a surprising phenomenon on \emph{steps}. Since the maximum order complexity is integer-valued and an increasing function, it is a step function. Each time $M(\mathcal{S},N)$ has a different value from $M(\mathcal{S},N-1)$, we say that $N$ is \emph{a step}. The ratio of successive steps is the ratio $N_1/N_2$ for two steps $N_1,N_2$ such that there is no $N\geq 1$ with $M(\mathcal{S},N_1)<M(\mathcal{S},N)<M(\mathcal{S},N_2)$.

We observed that the ratio of successive steps seems to converge.

\begin{conjecture}\label{Ratio_conjecture}
The ratio of successive steps for the maximum order complexity of a sequence related to the Thue--Morse sequence, respectively, the sequence related to the Zeckendorf sum of digits, tends to 2, respectively, the golden ration $\varphi$.
\end{conjecture}

This phenomenon seems to be related to numeration systems, and not directly to the fact that a sequence is automatic or morphic.

\begin{table}[!htp]
\begin{minipage}[b]{0.45\linewidth}\centering
\begin{tabular}{ll}
\noalign{\smallskip}
$N$& Ratio of steps in in Fig. 1. \\ \hline
2	 & - \\
5	& 2,5\\
12 & 2,4 \\
19 & 1,583333333 \\
30 & 1,578947368 \\
48 & 1,6 \\
77 & 1,604166667 \\
124 &	1,61038961 \\
200	& 1,612903226 \\
323	& 1,615 \\
522	& 1,616099071 \\
844	& 1,616858238 \\
1365 & 1,617298578 \\
2208 & 1,617582418 \\
3572 & 1,617753623 \\
5779 & 1,617861142 \\
9350 & 1,617926977 \\
15128 & 1,617967914 \\
24477	& 1,617993125 \\
39604 & 1,618008743 \\
64080 &	 1,618018382 \\
103683 & 1,618024345 \\
167762 & 1,618028028 \\
271444 & 1,618030305 \\
439205 & 1,618031712 \\
710648	& 1,618032582 \\
\noalign{\smallskip}\hline
\end{tabular}\label{table1}

\end{minipage}
\hspace{0.5cm}
\begin{minipage}[b]{0.45\linewidth}
\centering
\begin{tabular}{ll}
\noalign{\smallskip}
$N$& Ratio of steps in Fig. 2.\\ \hline
2 & - \\
4 & 2 \\
10 & 2,5 \\
23 & 2,3 \\
37 & 1,608695652 \\
52 & 1,405405405\\
73 & 1,403846154 \\
138 & 1,890410959 \\
270 & 1,956521739 \\
530 & 1,962962963 \\
1048 & 1,977358491 \\
2082 & 1,986641221 \\
4146 & 1,991354467 \\ 
8258 & 1,991799325 \\ 
16478 & 1,995398402 \\
32898 & 1,996480155 \\
65720 & 1,997689829 \\
131330 & 1,998326233 \\
262510 & 1,998857839 \\
524802 & 1,999169555 \\
1049302 & 1,999424545 \\
2098178 & 1,999594016 \\
4195754 & 1,999713084 \\
8390658 & 1,999797414 \\
\noalign{\smallskip}\hline
\end{tabular}\label{table2}

\end{minipage}
\end{table}

\begin{figure}[H]
\centering
\includegraphics[width=0.75\textwidth]{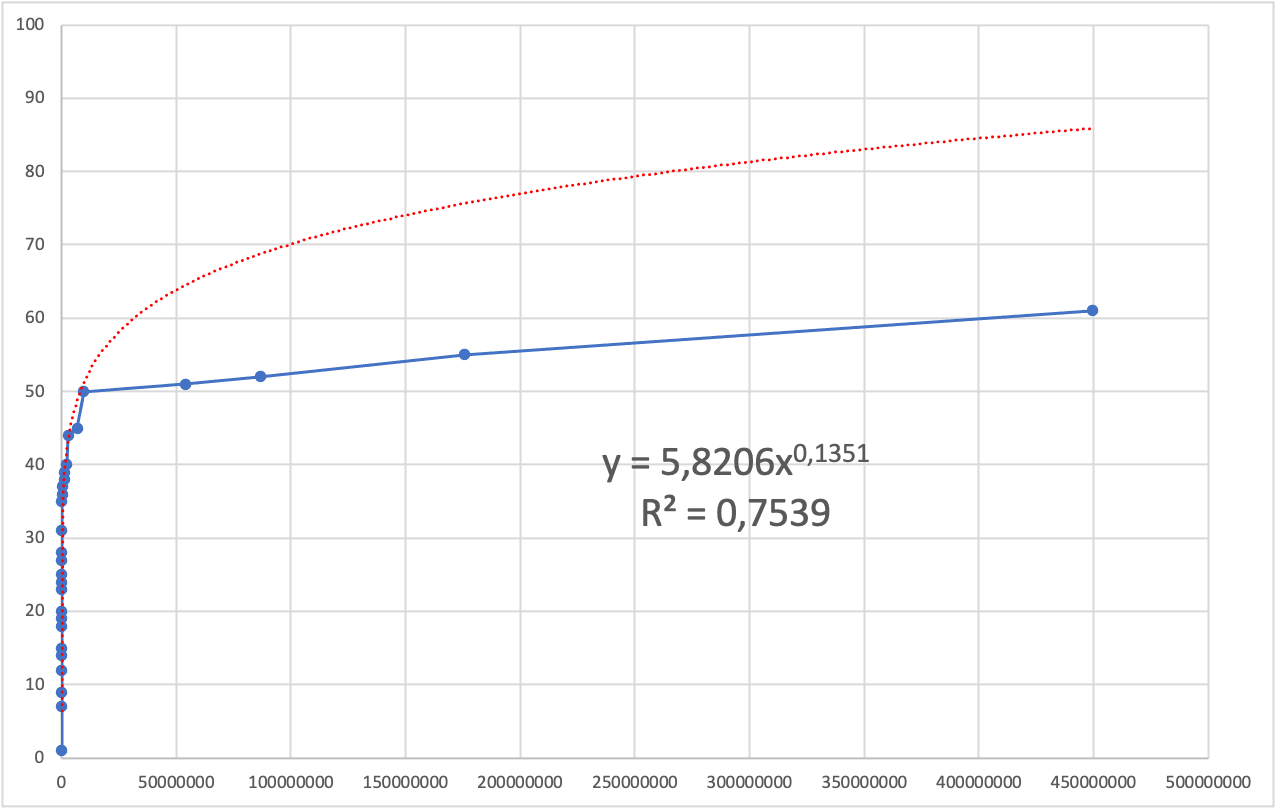}
\caption{Maximum order complexity of $\mathcal{S}_Z$  along cubes.}
\label{Zeckendorf_cubes}
\end{figure}

\begin{center}

\begin{minipage}[b]{0.45\linewidth}\centering
\begin{tabular}{ll}
\noalign{\smallskip}
$N$& Ratio of steps in Fig. 4. \\ \hline
2	& - \\
4&	2\\
6	&1,5\\
8	&1,33333333\\
18	&2,25\\
24	&1,33333333\\
42	&1,75\\
59	&1,40476191\\
171	&2,89830509\\
448	&2,61988304\\
18367&	40,9977679\\
22118&	1,20422497\\
28371&	1,28271091\\
83517	&2,94374538\\
167773&	2,00884850\\
317825&	1,89437514\\
439221&	1,38195863\\
832054&	1,89438574\\
1346308&	1,61805364\\
2178338	&1,61800866\\
3524634&	1,61803816\\
5702922&	1,61801821\\
9227522&	1,61803405\\
14930387&	1,61802779\\
24157917	&1,61803689\\
39088244	&1,61803040\\
63246131	&1,61803460\\
102334245&	1,61803170\\
165580287&	1,61803399\\
267914386&	1,61803311\\
433494698&	1,61803442\\
701408926&	1,61803346\\
\noalign{\smallskip}\hline
\end{tabular}\label{table3}
\end{minipage}

\end{center}

\begin{center}

\begin{minipage}[b]{0.45\linewidth}
\centering
\begin{tabular}{ll}
\noalign{\smallskip}
$N$& Ratio of steps in Fig. 5.\\ \hline
2	& - \\
9 &	4,5\\
37&	4,11111111\\
59	&1,59459460\\
67	&1,13559322\\
167	&2,49253731\\
273	&1,63473054\\
1677&	6,14285714\\
1960&	1,16875373\\
2178&	1,11122449\\
6378&	2,92837466\\
11095&	1,73957353\\
17046	&1,53636773\\
20398&	1,19664437\\
40879	&2,00406903\\
57958	&1,41779398\\
510998&	8,81669485\\
567415&	1,11040552\\
1356065	& 2,38989981\\
1390403	& 1,02532180\\
2264908	& 1,62895794\\
2860200 & 1,26283275\\
6783276 & 2,37160898\\
9626340 & 1,41912846\\
54170633	& 5,62733427\\
86568697	& 1,59807431\\
175721520	& 2,02985058\\
449510050 & 2,55808196\\
\noalign{\smallskip}\hline

\end{tabular}\label{table4}
\end{minipage}

\end{center}

The plot for cubes, see \Cref{Zeckendorf_cubes}, does not enlighten \textit{Conjecture 3.3} and \textit{Conjecture 3.4}. With our program and our machine, it is not possible to compute the maximum order complexity of a sequence any further than $10^9$ terms. Nevertheless, we believe that if it were possible to compute for a few more terms, both of theses conjectures should appear more clearly.

To perform all these plots, we used the cluster \emph{yeti} that consists of 4 nodes, 4 x Intel Xeon Gold 6130	CPU and 16 cores / CPU, with 768 GiB of memory, see \url{https://www.grid5000.fr/w/Grenoble:Hardware#yeti}.

~\newline

\textbf{Acknowledgements:} We would like to thank Arne Winterhof for very helpful discussions and comments on a preliminary version of the paper, and for sending us the bachelor thesis of his student Jan-Michael Holzinger. This work was supported partly by the French PIA project ``Lorraine Universit\'{e} d'Excellence'', reference ANR-15-IDEX-04-LUE, and by the projects ANR-18-CE40-0018 (EST) and ANR-20-CE91-0006 (ArithRand).

\bibliographystyle{spmpsci}
\bibliography{biblio} 

\end{document}